\documentclass[11pt, reqno]{amsart}
\usepackage{amssymb, latexsym, graphicx, mathrsfs, enumerate}
\usepackage[dvips]{color}
\usepackage[nohug,heads=vee]{diagrams}
\diagramstyle[labelstyle=\scriptstyle]

\let\oldtocsection=\tocsection

\let\oldtocsubsection=\tocsubsection

\let\oldtocsubsubsection=\tocsubsubsection

\renewcommand{\tocsection}[2]{\hspace{0em}\oldtocsection{#1}{#2}}
\renewcommand{\tocsubsection}[2]{\hspace{1em}\oldtocsubsection{#1}{#2}}
\renewcommand{\tocsubsubsection}[2]{\hspace{2em}\oldtocsubsubsection{#1}{#2}}

\newlength{\margins}
\usepackage[top=\margins,bottom=\margins,left=\margins,right=\margins]{geometry}

\numberwithin{equation}{section}
\newtheorem{Ass}[subsection]{Assumption}
\newtheorem{Thm}[subsection]{Theorem}
\newtheorem{Prop}[subsection]{Proposition}
\newtheorem{Cor}[subsection]{Corollary}
\newtheorem{Lem}[subsection]{Lemma}
\newtheorem{Conj}[subsection]{Conjecture}

\theoremstyle{definition}
\newtheorem{Def}[subsection]{Definition}
\newtheorem{Exa}[subsection]{Example}
\newtheorem{Rmk}[subsection]{Remark}

\begin{document}

\title[A uniform construction of smooth integral models and local densities]
{A uniform construction of smooth integral models and a recipe for computing local densities}

\author[Sungmun Cho]{Sungmun Cho}
\address{Department of
Mathematics, University of Toronto,  CANADA}
\email{sungmuncho12@gmail.com}

\maketitle

\begin{abstract}
In this paper, we explain a  simple and uniform construction of a smooth integral model associated to
a quadratic, (anti)-hermitian, and (anti)-quaternionic hermitian  lattice defined over an arbitrary  local field.
As one major application, we explain a recipe for computing local densities case by case,
which is an essential factor in the classification of  forms as above over the ring of integers of a number field,
 by introducing one conjecture about the number of rational points of the special fiber of a smooth integral model.
\end{abstract}


\section{Introduction}
\subsection{Introduction}
A long standing central problem in the arithmetic theory of hermitian (or quadratic) forms is the classification of hermitian (or quadratic) forms over the ring of integers of a number field.
If we let $R$ be the ring of integers of a number field $k$,
then for  a totally definite  hermitian (or quadratic) $R$-lattice $(L, H)$,
the genus of $(L, H)$, denoted by $\mathrm{gen}(L, H)$, is defined as  the set of (equivalence classes of)  hermitian (or quadratic) lattices
 that are locally equivalent to $(L, H)$.
 Since the local-global principle does not hold for a hermitian (or quadratic) lattice $(L, H)$, the set $\mathrm{gen}(L, H)$ is not trivial in general.
 However, it is well-known that  $\mathrm{gen}(L, H)$ is a finite set.
The computation of the total mass of $(L, H)$ ($=\sum_{(L^{\prime}, H^{\prime}) \in \mathrm{gen}(L, H)} \frac{1}{\#\mathrm{Aut}_{R}(L^{\prime}, H^{\prime})}$)
 is an essential ingredient for enumerating all elements of the set $\mathrm{gen}(L, H)$ explicitly.
The total mass of $(L, H)$ can be expressed as a product of local factors, the so-called local densities, by the celebrated Smith-Minkowski-Siegel mass formula.

The local density is defined as follows.
To simplify notation, we  consider a quadratic $A$-lattice  $(L, h)$ at this moment.
 Here, $A$ is the ring of integers of a local field $F$ with a uniformizer $\pi$ and the residue field $\kappa$.
 The local density was originally defined as the  limit of a certain sequence (\cite{K}), which is described below:
$$ \beta_{L}=\frac{1}{2}\cdot \lim_{N\to\infty} q^{-N \mathrm{dim}G} \# \underline{G}^{\prime}(A/\pi^NA). $$
Here $\underline{G}^{\prime}$ is a naive integral model of  the orthogonal group $\mathrm{O}(V, h)$, where $V=L\otimes_AF$, such that
$\underline{G}^{\prime}(R)=\mathrm{Aut}_{R}(L\otimes_AR, h\otimes_AR)$ for any commutative $A$-algebra $R$,
and $q$ is the cardinality of $\kappa$.
Note that this limit stabilizes after finite steps, and the formula is found in \cite{CS} and \cite{HS} for $\mathbb{Z}_p$.
 However, for a given quadratic lattice defined over a finite (especially ramified) extension of $\mathbb{Z}_2$,
 it is a nontrivial task to compute the above limit.
 Later, W. T. Gan and J.-K. Yu  found another approach for computing  local densities in \cite{GY}.
(We explain their approach here briefly
and for a detailed exposition  including a general lattice such as a hermitian lattice,  see Section 3 of \cite{GY}.)
 The local density of a quadratic $A$-lattice  $(L, h)$ can also be defined as an integral of a certain volume form
  $\omega^{\mathrm{ld}}$ associated to $\underline{G}^{\prime}$, which is described below:
 $$ \beta_{L}=\frac{1}{2}\int_{\mathrm{Aut}_{A}(L, h)} |\omega^{\mathrm{ld}}|.$$
On the other hand, the following integral is well known.
$$\int_{\mathrm{Aut}_{A}(L, h)} |\omega^{\mathrm{can}}|= q^{- \mathrm{dim}G}\cdot \# \underline{G}(A/\pi A).$$
Here $\underline{G}$ is the unique smooth affine group scheme (called a smooth integral model) of  $\mathrm{O}(V, h)$
such that
$\underline{G}(R)=\underline{G}^{\prime}(R)$
for any \'etale $A$-algebra $R$ (Proposition 3.7 in \cite{GY})
and $\omega^{\mathrm{can}}$ is a volume form associated to $\underline{G}$.
Therefore, in order to obtain an explicit formula for the local density, it suffices to
\begin{itemize}
\item determine $\underline{G}$ and the number of rational points of the special fiber of $\underline{G}$ (i.e. $\# \underline{G}(A/\pi A)$);
\item compare two volume forms $\omega^{\mathrm{ld}}$ and $\omega^{\mathrm{can}}$.
\end{itemize}
 The difference between two volume forms is direct from the construction of the smooth integral model $\underline{G}$.
 Therefore constructing the smooth integral model $\underline{G}$ and investigating the special fiber of $\underline{G}$ will lead us to an explicit formula.

The local density formula together with a smooth integral model $\underline{G}$ and its special fiber has been fully studied in \cite{GY} $(p\neq 2)$, \cite{C1}
(quadratic lattice with $A/\mathbb{Z}_2$ unramified),
and \cite{C2} (ramified hermitian lattice with $A/\mathbb{Z}_2$ unramified).
Indeed, the constructions of $\underline{G}$ in \cite{C1} and \cite{C2}, when $A/\mathbb{Z}_2$  is unramified,
are much more complicated than that of \cite{GY} when $p\neq 2$.
More explicitly, if $p\neq 2$,  then the construction of $\underline{G}$ is  based on the dual lattice $L^{\#}$ of $L$. 
In the case that  $A/\mathbb{Z}_2$  is unramified, one needs sublattices $A_i, B_i, W_i, X_i, Y_i, Z_i$ of $L$, for every nonnegative integer $i$'s, in order to construct $\underline{G}$.
Therefore, if $A/\mathbb{Z}_2$  is ramified, then it is natural to expect that one needs much more involved sublattices of $L$ beyond $A_i, B_i, W_i, X_i, Y_i, Z_i$ in order to  construct  $\underline{G}$.
One can also expect that a construction of $\underline{G}$ in a quadratic lattice is different from that of a hermitian lattice, as in \cite{C1} and \cite{C2}.

In this paper, we explain a uniform construction of a smooth integral model $\underline{G}$ associated to
a quadratic or a hermitian lattice  over any non-Archimedean local field, 
 based only on the dual lattice $L^{\#}$ of $L$. 
Our construction of $\underline{G}$  is   simple and canonical (independent of the choice of basis of $L$).
As one major application, we obtain the following theorem about the local density for a given lattice $(L, h)$:
\begin{Thm}(Theorem 4.3)
Let $\tilde{G}$ (resp. $G$) be the special fiber (resp. generic fiber) of $\underline{G}$ and let $q$ be the cardinality of $\kappa$.
Then the local density of ($L,h$) is
$$\beta_L=\frac{1}{[G:G^o]}q^N \cdot q^{-\mathrm{dim~} G} \cdot \#\tilde{G}(\kappa),$$
where $G^o$ is the identity component of $G$ and $N$ is a suitable  integer.
\end{Thm}
Therefore, in order to compute the local density after constructing $\underline{G}$,
we only have to know two ingredients in the above theorem; $q^N$ and $\#\tilde{G}(\kappa)$.
The ingredient $q^N$ is a quantity of the difference between two volume forms  $\omega^{\mathrm{ld}}$ and $\omega^{\mathrm{can}}$
and this can be computed easily based on the construction of $\underline{G}$, as explained in Theorem 4.3.
Thus, once we construct $\underline{G}$, the remaining challenging ingredient is the computation of $\#\tilde{G}(\kappa)$.

 For a given quadratic lattice $(L, h)$,
 let $L=\bigoplus_{i\geq 0}L_i$ be a Jordan splitting for $(L, h)$.
 Then we have a group homomorphism
$\varphi_{i, \kappa} : \tilde{G}(\kappa) \rightarrow \mathrm{O}(\bar{V_i}, \bar{h_i})(\kappa).$ (see Section 5.9)
Here, $\bar{V_i}$ and $\bar{h_i}$ are defined at the beginning of Section 5.9.
Let $$\varphi_{\kappa}=\prod_i \varphi_{i, \kappa} : \tilde{G}(\kappa) \rightarrow \prod_i \mathrm{O}(\bar{V_i}, \bar{h_i})(\kappa).$$
Then it seems to us that $\#\tilde{G}(\kappa)$ can be computed by proving the following conjecture.
 \begin{Conj}(Conjecture 5.10)
\[\textit{$\varphi_{\kappa}$ is surjective and its kernel is isomorphic to $(\textbf{A}^{l}\times (\mathbb{Z}/2\mathbb{Z})^{\beta})(\kappa)$ as sets,}\]
for $l=\textit{dimension of $\tilde{G}$ - $\sum_i$ (dimension of $\mathrm{O}(\bar{V_i}, \bar{h_i})$)} $ and for a certain non-negative integer $\beta$.
\end{Conj}
 Note that the dimension of $\tilde{G}$ is the same as the dimension of $G$ which is $n(n-1)/2$, where $G$ is the generic fiber of $\underline{G}$ and $n$ is the rank of $L$, since $\underline{G}$ is flat over $A$.
Once one verifies this conjecture, we obtain that
$$\# \tilde{G}(\kappa)=\prod_i\#(\mathrm{O}(\bar{V_i}, \bar{h_i})(\kappa))\cdot q^l\cdot 2^{\beta}.$$
Here, $q$ is the cardinality of $\kappa$.


It seems unlikely that the conjecture is solved in the general case.
In \cite{C1} and \cite{C2}, we proved this conjecture by writing down
formal matrix forms of an element of $\tilde{G}(\kappa)$ and the group homomorphism $\varphi_{\kappa}$ explicitly
based on a matrix form of an element of  $\underline{G}(A)$ for an unramified extension of $\mathbb{Z}_2$.
In the general case, it seems unlikely that such formal matrix forms can be written
because the  classification of hermitian (or quadratic) lattices over a ramified extension of $\mathbb{Z}_2$ is a highly involved problem.  (\cite{J}, \cite{O1}, \cite{O2}).

However, for a given lattice $L$, the construction of $\underline{G}$ studied in this paper
gives  an explicit matrix form of an element of $\underline{G}(A)$ which gives
explicit formal matrix forms of $\varphi_{\kappa}$ as well as an element of $\tilde{G}(\kappa)$, case by case.
In Remark 5.11, we explain a possible framework to prove the conjecture case by case, based on the proof of the conjecture on a specific case as given in
Theorem 5.4 (unimodular quadratic lattices with odd rank).
In order to show how the framework in Remark 5.11 works case by case,  we provide one another example (non-unimodular quadratic lattice) and prove the conjecture in this case in Example 5.12.

Note that this paper does not render the papers \cite{GY}, \cite{C1}, \cite{C2} obsolete since   we do not offer any new method to prove the conjecture,
which is one of main contributions in those papers.

Consequently, by assuming that the above conjecture can be proved case by case, this paper gives a new, simple and explicit recipe for computing the local density of a  given quadratic (or hermitian) lattice case by case,
especially, defined over an arbitrary ramified extension of $\mathbb{Z}_2$.
One strength of this recipe is that all computations for local densities of  given lattices are reduced to computations on  finitely generated $A$-modules,
limiting the language of schemes to  as little as possible.

This paper is organized as follows.
After fixing the notations in Section 2, we explain a uniform construction of a smooth integral model $\underline{G}$ in Section 3.
The constructions of $\underline{G}$ in \cite{GY}, \cite{C1}, \cite{C2} are special and simple cases (at most $\alpha=2$) of the construction studied in this paper,
which is explained in Section 4.1.
Then we give the local density formula in Section 4.2.
In order to illustrate how effective and useful the recipe computing local densities is,
we compute local densities of unimodular (i.e. $L^{\#}=L$) quadratic lattices with odd rank as an example in Section 5.
Finally, we explain a recipe for computing local densities by introducing Conjecture 5.10 and Remark 5.11
about the number of rational points of the special fiber of $\underline{G}$.


We note that our method also works for symplectic lattices, hermitian lattices, anti-hermitian lattices, quaternionic hermitian lattices,
and anti-quaternionic hermitian lattices over any non-Archimedean local field. 

\subsection{Acknowledgements}
The author had a motivation and an initial idea by hearing the discussion of Andrew Fiori and Gabriele Nebe, and by discussing with Brian Conrad
during the AIM workshop ``Algorithms for lattices and algebraic automorphic forms", May 2013.
The author would like to show deep appreciation to Wai Kiu Chan, Brian Conrad, Andrew Fiori,   Abhinav Kumar, Gabriele Nebe,
Rudolf Scharlau, Rainer Schulze-Pillot, and John Voight for valuable comments and discussions during the workshop.
The author also thanks the organizers of the workshop and the AIM for inviting him and for their hospitality.

Section 3, which is a main section of this paper, was inspired by a discussion with Radhika Ganapathy.

This paper was improved by discussing with Wee Teck Gan  during the author's visit to 2013 Pan Asian Number Theory (PANT) conference, Hanoi, July 2013, Vietnam.
He thanks the organizers of 2013 PANT conference and VIASM for inviting him and for their hospitality as well.

The author would like to express his deep gratitude to Juan Marcos Cervi\~no, Wai Kiu Chan, Brian Conrad,
Wee Teck Gan, Radhika Ganapathy, Benedict Gross, Manish Mishra, Gopal Prasad, Rudolf Scharlau, Rainer Schulze-Pillot, Olivier Ta\"ibi, Sandeep Varma, Jiu-Kang Yu,
and his wife Bogume Jang
for their encouragement and support on various aspects of the paper and the author's mathematical life.

\section{Notations}
This section is taken from \cite{GY}.
\subsubsection{} Let $F$ be a non-Archimedean local field 
with $A$ its ring of integers and $\kappa$ its residue field.
Let $p$ be the residue characteristic of $F$.
Choose a uniformizing element $\pi$ of $A$.
\subsubsection{} Let $(K, \sigma)$ be one of the following $F$-algebras with involution:
\begin{itemize}
\item $K=F$ with $\textit{char } F\neq 2$, $\sigma=$identity;
\item $K=E$, a separable quadratic extension, $\sigma=$the unique nontrivial automorphism of $E/F$;
\item $K=F\oplus F$, $\sigma(x,y)=(y,x)$;
\item $K=D$, the quaternion algebra over $F$, $\sigma=$the standard involution.
\end{itemize}
In the case of
$\textit{$K=M_2(F)$, the algebra of $(2 \times 2)$-matrices, $\sigma\begin{pmatrix} a&b\\ c&d \end{pmatrix}=\begin{pmatrix} d&-b\\ -c&a \end{pmatrix}$},$
 it is best handled by Morita context, as mentioned in the first paragraph of Section 4 of \cite{GY}, and is well explained in Section 8 of \cite{GY}.
Thus  we do not consider this case in our paper.

\subsubsection{} Let $B$ be a maximal $A$-order in $K$.
Then $B$ is uniquely determined. 
If $K=E$ is a ramified quadratic extension of $F$, or if $K=D$, we let $\pi_K$ be a uniformizer of $K$; 
in all other cases, we let $\pi_K=\pi$. 
\subsubsection{} Let $\epsilon$ be either $1$ or $-1$.
The triple $(K, \sigma, \epsilon)$ will be fixed throughout this paper, and by a hermitian form we always mean a $(\sigma, \epsilon)$-hermitian form.
If $(K, \epsilon, p)\neq (F, 1, 2)$, we consider a $B$-lattice $L$ (i.e., a free \textit{right} $B$-module of finite rank) of rank $n$
with a hermitian form $h : L \times L \rightarrow B.$
Our convention is
$$h(v\cdot a, w\cdot b)=\sigma(a)h(v,w) b, ~~~~~ h(w, v) = \epsilon \sigma (h(v, w)). $$
If $(K, \epsilon, p)= (F, 1, 2)$, we consider a $B (=A)$-lattice $L$ with a quadratic form $h : L \rightarrow B$,
and we say that
$h(v, w)=1/2\cdot (h(v+w)-h(v)-h(w))$.
We assume that $V=L\otimes_AF$ is nondegenerate with respect to $h$. 
The right $B$-module $L$ is also regarded as a left $B$-module by the rule $a\cdot v=v\cdot \sigma(a)$.

\begin{Def}
We define the dual lattice of $L$, denoted by $L^{\#}$, as
 $$L^{\#}=\{x \in L\otimes_A F : h(x, L) \subset B \}.$$
\end{Def}
The pair ($L, h$) is fixed throughout this paper.
\begin{Ass}
In this paper, we assume the following identity:
$$(L^{\#})^{\#}=L.$$
This assumption is true at least in the following cases:
\begin{itemize}
\item $K=F$, $K=E$, $K=F\oplus F$; 
\item $K=D$, the quaternion algebra over $F$, $\epsilon=1$.
\end{itemize}
The proof of the assumption in the above cases follows easily from the existence of a \textit{Jordan splitting} of a hermitian lattice $(L, h)$.
A \textit{Jordan splitting} of $(L, h)$ is  an orthogonal decomposition $L=\bigoplus_{i\geq 0}L_i$ such that $L^{\#}=\bigoplus_{i\geq 0}\pi_K^{-i}L_i$.
For a detailed explanation  of \textit{Jordan splitting}, see  \cite{O1} or \cite{O2} ($K=F$ with $\epsilon=1$),
\cite{J} and \cite{Sa} ($K=E$ with $\epsilon=1$), and
 \cite{GY} (all other cases except for $K=E$ with $\epsilon=-1$). 

If $K=E$ with $\epsilon=-1$, let $a\in F$ such that $E=F(\sqrt{a})$.
Then $(L, \sqrt{a}\cdot h)$ is a hermitian lattice with $\epsilon=1$ so that there is a Jordan splitting.
This gives a Jordan splitting for $(L, h)$, the indices being shifted according to the valuation of $1/\sqrt{a}$.
(Olivier Ta\"ibi provided the above proof to the author.)

For the following remaining case,
\begin{itemize}
\item $K=D$ with $\textit{char } F= 0$, the quaternion algebra over $F$,  $p=2$, $\epsilon=-1$,
\end{itemize}
we could not find a reference to mention whether or not the assumption is true.
If  there exists a \textit{Jordan splitting} of a hermitian lattice $(L, h)$  in the remaining  case 
so that the assumption is true, as we highly expect,
 then all contents of this paper hold for such case.
\end{Ass}



\subsubsection{} Let $G$ be the reductive algebraic group over $F$ such that
$$G(E)=\mathrm{Aut}_{K\otimes_FE}(V\otimes_FE, h\otimes_FE)$$
for any commutative $F$-algebra $E$.
Then $G$ is a classical group, not necessarily connected.
Indeed, it is essential that $G$ is smooth in our construction of the smooth integral model $\underline{G}$ due to Theorem 3.1 below.
Thus we exclude the case in Subsection 2.0.2 that $G$ is an orthogonal group over a field of characteristic 2, i.e.,
the case $K=F$ with $\textit{char } F= 2$ and $\sigma=$identity.

We denote by $\underline{\mathrm{GL}_K(V)}_{/F}$ the $F$-group scheme whose group of $E$-valued points is
$\mathrm{GL}_{K\otimes_FE}(V\otimes_FE)$
for any commutative $F$-algebra $E$.
If $K$ is commutative, this is just the Weil restriction $\mathrm{Res}_{K/F}\mathrm{GL}_K(V)$.

\section{Smooth integral model}

We start with the following theorem which  is crucially used in this paper.
\begin{Thm}(Proposition 3.7 in \cite{GY})
Assume that $A$ is a discrete valuation ring.
Let $\underline{G}^{\prime}$ be an affine group scheme over $A$ of finite type with smooth generic fiber $G$.
Then there exists a unique smooth affine group scheme $\underline{G}$ over $A$ with generic fiber $G$ such that
$$\underline{G}(R)=\underline{G}^{\prime}(R)\textit{     for any \'etale $A$-algebra $R$.}$$
\end{Thm}

 Let $\underline{G}^{\prime}$ be a naive integral model of $G$ such that
for any commutative $A$-algebra $R$,
$\underline{G}^{\prime}(R)=\mathrm{Aut}_{B\otimes_AR}(L\otimes_AR, h\otimes_AR)$.
Then, the above theorem tells the existence of a unique smooth integral model
 $\underline{G}$  of $G$
such that
$\underline{G}(R)=\underline{G}^{\prime}(R) \textit{     for any \'etale $A$-algebra $R$}$.
In this section, we give an explicit construction of the smooth integral model $\underline{G}$.

\subsection{Constructions of $T^0$ and $H^0$ based on \cite{GY}}
We first define a functor $T^0$ from the category of commutative flat $A$-algebras to the category of rings.
For any commutative flat $A$-algebra $R$, set
$$T^0(R) = \{X \in \mathrm{End}_{B\otimes_AR}(L \otimes_A R) : X(L^{\#}\otimes_AR) \subset L^{\#}\otimes_AR\}.$$
Clearly, the set $T^0(R)$ is closed under addition and multiplication, so $T^0(R)$ has the structure of a ring.
The functor $T^0$ is representable by a flat $A$-algebra which is a polynomial ring over $A$ of $n^2\cdot[K:F]$ variables,
and this is explained in Section 5.2 of \cite{GY}.
Therefore, $T^0$ has the structure of a scheme of rings.

For the future use, we state the following proposition.
\begin{Prop}
For a flat $A$-algebra $R$, choose $X, Y \in T^0(R)$.
We define the adjoint $X^{ad}$ of $X$ characterized as
$h(X(v), w)=h(v, X^{ad}(w))$, where $v, w \in L \otimes_A R$.
Then $X^{ad}, ~~~ X^{ad}\cdot Y\in T^0(R)$ and $(X^{ad})^{ad}=X$.
\end{Prop}
\begin{proof}
It is clear that $X^{ad}$ stabilizes $L^{\#}$ as well as $L$, by using  Assumption 2.2, $(L^{\#})^{\#}=L$.
\end{proof}

We  define another functor $H^0$ from the category of commutative flat $A$-algebras to the category of abelian groups.
For any commutative flat $A$-algebra $R$, set
 $$H^0(R) = \{\textit{$f$ : $f$ is an hermitian form on $L\otimes_AR$ such that $f(L\otimes_AR, L^{\#}\otimes_AR)\subset B\otimes_AR$}\}.$$
The functor $H^0$ is also representable by a flat $A$-algebra which is a polynomial ring over $A$ of $n^2\cdot[K:F]-\mathrm{dim~}G$ variables,
and this is explained in Section 5.4 of \cite{GY}.
Note that the fixed hermitian form $h$ is an element of $H^0(A)$.

\subsection{Construction of three maps $\varphi_{0, R}$, $\psi_{0, R}$, $\overline{\varphi_{0, R}}$ }
  \subsubsection{}  We define a map $\varphi_{0, R}$ from $T^0(R)$ to $H^0(R)$ for a flat $A$-algebra $R$ as follows:
  $$ \varphi_{0, R} : T^0(R) \longrightarrow H^0(R), ~~~~~ X\mapsto h\circ X.$$
  Here, $h\circ X$ is a hermitian form on $L\otimes_AR$ such that $h\circ X(v, w)=h(X(v), X(w))$ for $v, w \in L\otimes_AR$.
  Then this map is represented by a morphism of schemes, denoted by $\varphi_0$.


  \subsubsection{}
   Let $R$ be a commutative $A$-algebra.
   Since $T^0$ is representable by a polynomial, the fiber of the identity along with the morphism $T^0(R[\varepsilon]/\varepsilon^2) \rightarrow T^0(R)$
   is naturally identified with $T^0(R)$.
   Similarly, the fiber of the fixed hermitian form $h$ along with the morphism $H^0(R[\varepsilon]/\varepsilon^2) \rightarrow H^0(R)$
   is  identified with $H^0(R)$.
   Then the morphism $\varphi_0$ induces another map $\psi_{0, R}$ from  $T^0(R)$ to $H^0(R)$ as follows:
   $$\psi_{0, R} : T^0(R) \longrightarrow H^0(R), \textit{     } \psi_{0, R} (X)(v, w) = h(v, X(w))+h(X(v), w)).$$
    Note that both $T^0(R)$ and $H^0(R)$ have $R$-module structures and  $\psi_{0, R}$ is an $R$-module homomorphism.
    Like as $\varphi_0$, the map $\psi_{0, R}$ is also represented by a morphism of schemes, denoted by $\psi_0$.
  Indeed, $\psi_0$ is the differential of $\varphi_0$.

\subsubsection{}    We consider the following quotient map
$$\overline{\varphi_{0, R}} : T^0(R) \longrightarrow H^0(R)\longrightarrow H^0(R)/\mathrm{Im~}\psi_{0, R}, \textit{   for a flat $A$-algebra $R$}$$
induced from $\varphi_{0, R}$.
Notice that $H^0(R)/\mathrm{Im~}\psi_{0, R}$ is an abelian group since it is an $R$-module. 
The map $\overline{\varphi_{0, R}}$ is then  a group homomorphism since
$\varphi_{0, R}(X+Y)=\varphi_{0, R}(X)+\varphi_{0, R}(Y)+\psi_{0, R}(X^{ad}\cdot Y)$
for  $X, Y \in T^0(R)$.
Here, $X^{ad}\cdot Y \in T^0(R)$ by Proposition 3.3.
Note that $\overline{\varphi_{0, R}}$ is not an $R$-module homomorphism.

\subsection{Construction of $T^1$}
We define the functor $T^1$ from the category of commutative flat $A$-algebras to the category of abelian groups as follows:
\[
T^1(R)=\left\{
  \begin{array}{l l}
 \{X\in \mathrm{Ker~}\overline{\varphi_{0, A}} : X^{ad} \in \mathrm{Ker~}\overline{\varphi_{0, A}}\} & \quad \textit{if $R=A$};\\
 R\otimes_AT^1(A) & \quad \textit{if $R$ is a flat $A$-algebra}.
    \end{array} \right.
\]
Note that the set $\{X\in \mathrm{Ker~}\overline{\varphi_{0, A}} : X^{ad} \in \mathrm{Ker~}\overline{\varphi_{0, A}}\}$ is an $A$-module
so $T^1(A)$ is well-defined and therefore the functor $T^1$ is well-defined.
We can also easily see that $T^1(R)$ is an $R$-submodule of $T^0(R)$ for a flat $A$-algebra $R$.
The functor $T^1$ has the following property for an  \'etale $A$-algebra $R$ which is stated in Theorem 3.6.
Since our purpose is the construction of the smooth integral model $\underline{G}$
and it is characterized in terms of   \'etale $A$-algebras as mentioned in Theorem 3.1,
it is necessary to have  explicit descriptions of our functors defined in this section in terms of   \'etale $A$-algebras.

\begin{Thm}
Let $R$ be an  \'etale $A$-algebra.
Then it is easy to see that
$R\otimes_AT^1(A) \subseteq \{X\in \mathrm{Ker~}\overline{\varphi_{0, R}} : X^{ad} \in \mathrm{Ker~}\overline{\varphi_{0, R}}\}$
as $R$-submodules of $T^0(R)$.
Furthermore, these two submodules are same.
\end{Thm}
\begin{proof}
Let $R$ be    an \'etale local ring over $A$.
Note that such $R$ becomes finite over $A$ since
  any \'etale local ring $R$ over a henselian local ring is finite   by Proposition 4 of Section 2.3 in \cite{BLR} and $A$ is henselian by the assumption made in Subsection 2.0.1.
Then a  morphism $\textit{Spec R}\rightarrow \textit{Spec A}$ is a Galois covering (cf. Section 6.2, Example B in \cite{BLR}).
We remark that $T^0(R)=R\otimes_AT^0(A)$ and $H^0(R)=R\otimes_AH^0(A)$.
It is, then, easy to show that the set $\{X\in \mathrm{Ker~}\overline{\varphi_{0, R}} : X^{ad} \in \mathrm{Ker~}\overline{\varphi_{0, R}}\}$
is stabilized by an element of $\mathrm{Aut}_A(R)$.
Therefore, the above set descents to an $A$-submodule $\mathcal{M}$ of $T^0(A)$ by \textit{Galois Descent}.
Choose $X\in \mathcal{M}$.
Since $\varphi_{0, R}(X), \varphi_{0, R}(X^{ad})\in \mathrm{Im~}\psi_{0, R} \cap H^0(A)$,
 we can see that $\varphi_{0, A}(X), \varphi_{0, A}(X^{ad})\in \mathrm{Im~}(\psi_{0, A})$.
Thus $X\in T^1(A)$ so $\mathcal{M} \subseteq T^1(A)$.
This completes the proof. 
\end{proof}

We now claim that the functor $T^1$ is represented by a unique polynomial ring. More precisely,

\begin{Lem}
The functor $T^1$ is representable by a flat $A$-algebra which is a polynomial ring over $A$ of $n^2\cdot[K:F]$ variables.
\end{Lem}

\begin{proof}
Let $R$ be a flat $A$-algebra.
Since $T^1(R)=R\otimes_AT^1(A)$ is a finitely generated free $R$-module,
 $T^1$ is representable by a polynomial ring over $A$.
Thus the relative dimension of $T^1$ over $\textit{Spec A}$ is the same as the dimension of the generic fiber of $T^1$ over $\textit{Spec F}$,
which is the same as the dimension of $T^1(F)$ as an $F$-vector space.
We claim that $T^1(F)~~(=F\otimes_AT^1(A))~~=T^0(F)$.

It is easy to show that $F\otimes_AT^1(A)=\{X\in \mathrm{Ker~}\overline{\varphi_{0, F}} : X^{ad} \in \mathrm{Ker~}\overline{\varphi_{0, F}}\}$.
Since $\psi_{0, F}$ is surjective,
The latter is the same as $T^0(F)$, whose dimension as an $F$-vector space is $n^2\cdot [K:F]$.
\end{proof}

Note that there is a natural morphism from $T^1$ to $T^0$ mapping $X$ to $X$ where $X\in T^1(R)$ with a flat $A$-algebra $R$.
This morphism is represented by a morphism of schemes and we denote it by  $\iota_0$.
We note that $\iota_0$ gives a subfunctor on flat $A$-algebras, but not immersion as schemes.
For example, if $R$ is a torsion $A$-algebra, then $\iota_{0, R}$ is no longer injective.
Let $\varphi_1$ (resp. $\psi_1$) be the morphism from $T^1$ to $H^0$ induced by $\varphi_0$ (resp. $\psi_0$) composed with $\iota_0$
 as the following commutative diagram of morphisms of schemes:
\begin{diagram}
T^1\\
\dTo^{\iota_0}  &\rdTo^{\varphi_1, \psi_1}\\
T^0 &\rTo^{\varphi_0, \psi_0} & H^0
\end{diagram}

\subsection{Construction of the functor $T^{m+1}$ on the category of \'etale $A$-algebras}
We define the functor $T^{m+1}$, for all $m \geq 0$, from the category of \'etale $A$-algebras to the category of abelian groups as follows:
\[
T^{m+1}(R)=
 \{X\in \mathrm{Ker~}\overline{\varphi_{m, R}} : X^{ad} \in \mathrm{Ker~}\overline{\varphi_{m, R}}\}
 \]
with the following morphisms:
\[
\left\{
  \begin{array}{l }
  \iota_{m-1, R}: T^m(R)\longrightarrow T^{m-1}(R), ~~~~~ X \mapsto X;\\
 \varphi_{m, R}, \psi_{m, R} : T^m(R) \longrightarrow H^0(R), ~~~~~ \varphi_{m, R}=\varphi_{m-1, R}\circ\iota_{m-1, R},
 ~~\psi_{m, R}=\psi_{m-1, R}\circ\iota_{m-1, R};\\
 \overline{\varphi_{m, R}} : T^m(R) \longrightarrow H^0(R)\longrightarrow H^0(R)/\mathrm{Im~}\psi_{m, R}.
    \end{array} \right.
\]
We will show that the above functor is well-defined in the following theorem.
For convenience, let $T^{-m}=T^0$ and let $\varphi_{-m, R}=\varphi_{0, R}$ and $\psi_{-m, R}=\psi_{0, R}$ for all $m\geq 0$.


\begin{Thm}
Let $R$ be an \'etale $A$-algebra.
Then 
 $T^{m+1}(R)$ with morphisms $\iota_{m-1, R}, \varphi_{m, R}, \psi_{m, R}$, for all $m\geq 0$, is well-defined and is an additive group.
In particular, $X^{ad}\cdot Y\in T^m(R)$ with $X, Y\in T^m(R)$ and both $\psi_{m, R}$ and $\overline{\varphi_{m, R}}$   are  group homomorphisms.
\end{Thm}
\begin{proof}
We prove this by induction.
When $m=0$, the theorem follows from Proposition 3.3 and Section 3.4.3.  

Suppose that the theorem is true for all $n\leq m-1$, where $m \geq 1$, so that
$$T^{m}(R)=
 \{X\in \mathrm{Ker~}\overline{\varphi_{m-1, R}} : X^{ad} \in \mathrm{Ker~}\overline{\varphi_{m-1, R}}\}$$
 is well-defined.
Let $X, Y\in T^{m}(R)$ and choose $Z\in T^{m-1}(R)$ such that $\varphi_{m-1, R}(X^{ad})=\psi_{m-1, R}(Z)$.
 To show that $X^{ad}\cdot Y\in T^m(R)$,
we consider the following identity:
\begin{align*}
\begin{split}
\varphi_{m-1, R}(X^{ad}\cdot Y)=
\psi_{m-1, R}(Y^{ad}\cdot Z\cdot Y).
\end{split}
\end{align*}
Since $Z\cdot Y$ and $Y^{ad}\cdot Z\cdot Y \in T^{m-1}(R)$ by hypothesis of induction,
we  conclude that
 $\overline{\varphi_{m-1, R}}(X^{ad}\cdot Y)=0$.
Similarly, $\overline{\varphi_{m-1, R}}(Y^{ad}\cdot X)=0$  so $X^{ad}\cdot Y\in T^m(R)$.

To show that $\overline{\varphi_{m, R}}$   is a group homomorphism,
choose $X, Y \in T^m(R)$.
Let us observe the following identity:
  \begin{align*}
  \varphi_{m, R}(X+Y) 
  =\varphi_{m, R}(X)+\varphi_{m, R}(Y)+\psi_{m, R}(X^{ad}\cdot Y).
  \end{align*}
  Since
 $X^{ad}\cdot Y\in T^m(R)$, $\overline{\varphi_{m, R}}$   is a group homomorphism.
 This completes the proof.
 \end{proof}
 \begin{Cor}
 Let $R$ be an \'etale $A$-algebra.
Let  $X\in T^{m+1}(R)$ and let $Y\in T^{m}(R)$.
Then $X\cdot Y\in \mathrm{Ker~}\overline{\varphi_{m, R}}$.
 \end{Cor}
 \begin{proof}
 Let $\varphi_{m, R}(X)=\psi_{m, R}(Z)$ for some $Z\in T^m(R)$.
  Then $\varphi_{m, R}(X\cdot Y)=\psi_{m, R}(Y^{ad}\cdot Z\cdot Y)$ and $Y^{ad}\cdot Z\cdot Y\in T^m(R)$ by the above theorem.
 This completes the proof.
 \end{proof}

\subsection{Construction of the functors $T^{m+1}$ and $\widetilde{T}$ on the category of  flat $A$-algebras}
We extend the functor $T^{m+1}$ from the category of commutative flat $A$-algebras to the category of abelian groups as follows:
\[
T^{m+1}(R)=\left\{
  \begin{array}{l l}
 \{X\in \mathrm{Ker~}\overline{\varphi_{m, A}} : X^{ad} \in \mathrm{Ker~}\overline{\varphi_{m, A}}\} & \quad \textit{if $R=A$};\\
 R\otimes_AT^{m+1}(A) & \quad \textit{if $R$ is a flat $A$-algebra}.
    \end{array} \right.
\]

We can also show that  the set $\{X\in \mathrm{Ker~}\overline{\varphi_{m, A}} : X^{ad} \in \mathrm{Ker~}\overline{\varphi_{m, A}}\}$ is an $A$-module so
   $T^{m+1}(A)$ is well-defined and therefore $T^{m+1}$ is well-defined.
   In addition, $T^{m+1}(R)$ is an $A$-submodule of $T^{m}(R)$ for a flat $A$-algebra $R$.
   Furthermore, for an  \'etale $A$-algebra $R$, we can show that
   $$ R\otimes_AT^{m+1}(A) \left( = T^{m+1}(R) \right)=\{X\in \mathrm{Ker~}\overline{\varphi_{m, R}} : X^{ad} \in \mathrm{Ker~}\overline{\varphi_{m, R}}\},$$
   as submodules of $T^{m}(R)$,
   and its proof is similar to that of Theorem 3.6 with induction on $m$ so we skip it.

Define $\iota_{m-1, R} : T^m(R)\longrightarrow T^{m-1}(R), X \mapsto X$, for a flat $A$-algebra $R$.
Then we can easily show that $T^{m+1}$ is representable by a flat $A$-algebra which is a polynomial ring over $A$ of $n^2\cdot[K:F]$ variables
and that three morphisms $\iota_{m-1, R}, \varphi_{m, R}=\varphi_{m-1, R}\circ \iota_{m-1, R}$, and $\psi_{m, R}=\psi_{m-1, R}\circ \iota_{m-1, R}$,
 for a flat $A$-algebra $R$,
are represented by morphisms of schemes, denoted by $\iota_{m-1}, \varphi_{m}, \psi_{m}$ respectively.
The proof of these is similar to  that of Lemma 3.7 with induction on $m$ so we skip it.
Note that $T^m(F)=T^0(F)$ for all $m$.

Let us  observe the following sequence:
 $$T^0(R)\supseteq T^1(R)\supseteq \cdots \supseteq T^m(R)\supseteq  \cdots$$ for a flat $A$-algebra $R$.
 We emphasize that all ranks of $T^i(R)$, as finitely generated free $R$-modules, are same and it is $n^2\cdot[K:F]$.
Indeed, there is a suitable integer $\eta$ which makes the above sequence stabilize and this is proved in the following theorem.
\begin{Thm}
There exists an integer $\eta ~(\geq 0)$ such that $T^m=T^\eta$ for all $m \geq \eta$.
\end{Thm}
\begin{proof}
Let $M=T^0(A)/T^1(A)$.
 It  is then clear that $M$ is a torsion $A$-module since the rank of $T^0(A)\otimes_AF$ is the same as that of $T^1(A)\otimes_AF$.
  Let $l$ be the smallest non-negative integer such that $\pi^l\cdot M=0$.
Let $$T'(R)=\{\pi^{2l}\cdot X : X\in T^0(R)\}$$ for a flat $A$-algebra $R$. 
Then $T'$ is representable by  a flat $A$-algebra which is a polynomial ring over $A$.
Let $\varphi'$ and $\psi'$ be morphisms from $T'$ to $H^0$ induced from $\varphi_0$ and $\psi_0$, respectively.
We choose an element $\pi^{2l}\cdot X\in T'(R)$ for an \'etale $A$-algebra $R$.
 Since $\pi^l\cdot X\in T^1(R)$,
 $\varphi_{0, R}(\pi^l\cdot X)=\psi_{0, R}(Y)$
 for a certain $Y\in T^0(R)$.
Then, $$\varphi'_R(\pi^{2l}\cdot X)=\pi^{2l}\cdot \psi_{0, R}(Y)=\psi'_R(\pi^{2l}\cdot Y)$$ with $\pi^{2l}\cdot Y\in T'(R)$.
Therefore, $T'(R)\subset T^m(R)$ for every integer $m$, where $R$ is an \'etale $A$-algebra.

On the other hand, if $T^l=T^{l+1}$ for certain non-negative integer $l$, then it is clear that $T^l=T^{l'}$, for all $l'\geq l$.   
Therefore, the above sequence stabilizes and this completes the proof.
\end{proof}
Let $\alpha$ be the smallest non-negative integer satisfying that  $T^m=T^{\alpha}$ for all $m \geq \alpha$.
We finally define the functor $$\widetilde{T}:=T^{\alpha},$$ equipped with two morphisms $\varphi_{\alpha}$ and $\psi_{\alpha}$.
We note that Theorem 3.12 and its proof only assert the stability of the above sequence given before Theorem 3.12. 
 In order to find $\alpha$,
  one can construct $T^{m}(A)$ explicitly starting from $m=0$ and $\alpha$ is the first integer $m$ such that $T^{m}(A)=T^{m+1}(A)$
  (equivalently, the map $\overline{\varphi_{m, A}} : T^m(A) \rightarrow H^0(A)/\mathrm{Im~}\psi_{m, A}$ is zero).
Indeed, we expect that $\alpha$ is at most $e'+1$ where the ramification index $e=2e'$ or $e=2e'-1$, as this is true for unramified case in Section 4.1
and for two examples covered in Section 5.1 and Example 5.12.


\subsection{Construction of $\widetilde{H}$}

Recall that  $\psi_{\alpha, R} : \widetilde{T}(R) \longrightarrow H^0(R)$ is $R$-linear for a flat $A$-algebra $R$.
Define the functor $\widetilde{H}$ from the category of commutative flat $A$-algebras to the category of abelian groups as follows:
$$\widetilde{H}(R)=\mathrm{Im~}\psi_{\alpha, R}.$$
\begin{Thm}
The functor $\widetilde{H}$  is representable by a flat $A$-algebra which is a polynomial ring over $A$ of $n^2\cdot[K:F]-\mathrm{dim~}G$ variables
and the map $\psi_{\alpha, R} : \widetilde{T}(R) \rightarrow \widetilde{H}(R)$ is represented by a morphism of schemes, denoted by $\widetilde{\psi}$.

The map $\widetilde{\psi_{R_{\kappa}}} : \widetilde{T}(R_{\kappa}) \rightarrow \widetilde{H}(R_{\kappa})$ is surjective
for any $\kappa$-algebra $R_{\kappa}$.
\end{Thm}

\begin{proof}
 Since $\widetilde{T}(R)=R\otimes_A\widetilde{T}(A)$ and $\psi_{\alpha, R}$ is $R$-linear for a flat $A$-algebra $R$,
we have that $\widetilde{H}(R)=R\otimes_A\widetilde{H}(A)$.
Therefore, $\widetilde{H}$ is representable by a polynomial ring over $A$.
Since $\widetilde{H}(F)=H^0(F)$, the dimension of $\widetilde{H}(F)$ is $n^2\cdot [K:F]-\mathrm{dim~}G$
and this is the relative dimension of $\widetilde{H}$ over $\textit{Spec A}$.
The representability of $\psi_{\alpha, R}$ is obvious.

To show that $\widetilde{\psi_{R_{\kappa}}} : \widetilde{T}(R_{\kappa}) \rightarrow \widetilde{H}(R_{\kappa})$ is surjective for a $\kappa$-algebra
$R_{\kappa}$, we choose a flat $A$-algebra $R$ such that $R\otimes_A\kappa=R_{\kappa}$.
Then we have the following commutative diagram:
\begin{diagram}
\widetilde{T}(R)   &\rTo^{\widetilde{\psi_{R}}=\psi_{\alpha, R}} &\widetilde{H}(R)\\
\dTo_{} & &\dTo_{}\\
\widetilde{T}(R_{\kappa}) &\rTo^{\widetilde{\psi_{R_{\kappa}}}} &\widetilde{H}(R_{\kappa})
\end{diagram}

Since $\widetilde{T}$ and $\widetilde{H}$ are representable by affine spaces, two vertical maps are surjective.
One can also show surjectivity of two vertical maps by using Hensel's lemma
since $\widetilde{T}$ and $\widetilde{H}$ are smooth. 
In addition, $\widetilde{\psi_{R}}$ is surjective by construction.
Therefore, $\widetilde{\psi_{R_{\kappa}}}$ is surjective as well.
\end{proof}

\begin{Thm}
Let $R$ be a flat $A$-algebra.
Then the image of $\widetilde{T}(R)$, under the map $\varphi_{\alpha, R}$, is a subset of $\widetilde{H}(R)$.
Therefore, the morphism $\varphi_{\alpha} : \widetilde{T} \rightarrow H^0$ factors through $\widetilde{H}$.
\end{Thm}
\begin{proof}
Let $R$ be an \'etale $A$-algebra.
Suppose  that there is $X\in T^{\alpha}(R)=\widetilde{T}(R)$ such that $\varphi_{\alpha, R}(X) \notin \widetilde{H}(R)$.
Then $\overline{\varphi_{\alpha, R}}(X)\neq 0$ and so $X$ is not contained in $T^{\alpha+1}(R)$.
This  contradicts the  definition of $\alpha$. Thus $\varphi_{\alpha, R}(X) \in \widetilde{H}(R)$ for any $X\in \widetilde{T}(R)$.

Let $R$ be a flat $A$-algebra.
Choose $X\in T^{\alpha}(R)$.
We can write $X=\sum_ir_i\cdot X_i$ with $r_i\in R$ and $X_i\in T^{\alpha}(A)=\widetilde{T}(A)$.
Then $\varphi_{\alpha, R}(X)=\sum_ir_i^2\cdot \varphi_{\alpha, A}(X_i)+\sum_{i<j}r_ir_j\cdot \psi_{\alpha, A}(X_i^{ad}\cdot X_j)$.
Since $\varphi_{\alpha, A}(X_i)$, $\psi_{\alpha, A}(X_i^{ad}\cdot X_j) \in \widetilde{H}(A)$ and $\widetilde{H}(R)=R\otimes_A\widetilde{H}(A)$,
$\varphi_{\alpha, R}(X)$ is contained in $\widetilde{H}(R)$.
\end{proof}
Let $\widetilde{\varphi}$ be the morphism from $\widetilde{T}$ to $\widetilde{H}$ induced from $\varphi_{\alpha}$.
In conclusion, we have constructed two morphisms, namely  $\widetilde{\varphi}$ and $\widetilde{\psi}$, from $\widetilde{T}$ to $\widetilde{H}$.

\subsection{Construction of $\underline{G}$}
   We define two functors from the category of commutative flat $A$-algebras to the category of sets as follows:
     \[\underline{M}(R)=\{1+X : X\in \widetilde{T}(R)\},\]
     \[\underline{H}(R)=\{h+f : f\in \widetilde{H}(R)\}.\]
Then these two functors are  representable by flat $A$-algebras which are polynomial rings over $A$.
Note that $h$ is the fixed hermitian form throughout this paper.
 \begin{Lem}
The functor $\underline{M}$ is a functor from the category of commutative flat $A$-algebras to the category of monoids under multiplication.
      \end{Lem}

    \begin{proof}
Choose $1+X, 1+X' \in \underline{M}(R)$ for a flat $A$-algebra $R$.
Then it suffices to show that $(1+X)(1+X')-1\in \widetilde{T}(R)$.
Since $X\cdot X'\in \widetilde{T}(R)$ by Theorem 3.9 combined with the fact that $\widetilde{T}(R)=R\otimes_A\widetilde{T}(A)$ and $\widetilde{T}(R)$ is an additive group, we have that $X+X'+X\cdot X'\in \widetilde{T}(R)$.
\end{proof}

For any commutative  $A$-algebra $R$, set
 $$\underline{M}^{\ast}(R)=\{ m \in \underline{M}(R) :  \textit{there exists $m^{-1}\in \underline{M}(R)$ such that $m\cdot m^{-1}=m^{-1}\cdot m=1$}\}.$$
Then $\underline{M}^{\ast}$ is represented by a group scheme $\underline{M}^{\ast}$
and $\underline{M}^{\ast}$ is an open subscheme of $\underline{M}$, with generic fiber $M^{\ast}=\underline{\mathrm{GL}_K(V)}_{/F}$,
and thus $\underline{M}^{\ast}$ is smooth over $A$ since $\underline{M}$ is smooth over $A$.
The proofs of these 
 are similar to  Section 3.2 of \cite{C1}, by using the following Lemma
 so we skip them.

\begin{Lem}
For a flat $A$-algebra $R$, let $m\in \underline{M}(R)$ such that $m^{-1}\in \mathrm{Aut}_{B\otimes_AR}(L\otimes_AR)$.
Then such $m^{-1}$ is an element of $\underline{M}(R)$.
\end{Lem}

\begin{proof}
By Lemma 3.2 of \cite{C1}, $m^{-1}\in T^0(R)$.
Let $m=1+\sum_ir_i\cdot X_i$ and  $m^{-1}=1+\sum_is_j\cdot Y_j$ with $X_i\in \widetilde{T}(A)$, $Y_j\in T^0(A)$ and $r_i, s_j\in R$.
Then $\sum_is_j\cdot Y_j=-\sum_ir_i\cdot X_i-(\sum_ir_i\cdot X_i)\cdot (\sum_is_j\cdot Y_j)$
and $\sum_is_j\cdot Y_j^{ad}=-\sum_ir_i\cdot X_i^{ad}-(\sum_ir_i\cdot X_i^{ad})\cdot (\sum_is_j\cdot Y_j^{ad})$.
Thus, by Corollary 3.10 with an induction on $n$, we have that $\sum_is_j\cdot Y_j\in T^{n}(R)$ for all $n\geq 0$.
\end{proof}

\begin{Thm}
    For any flat $A$-algebra $R$, the group $\underline{M}^{\ast}(R)$ acts on the right of $\underline{H}(R)$
    by $(h+f)\cdot m=(h+f)\circ m $. Then this action is represented by an action morphism
     \[\underline{H} \times \underline{M}^{\ast} \longrightarrow \underline{H} .\]
      \end{Thm}

\begin{proof}
Let $R$ be a flat $A$-algebra.
It suffices to show that $(h+f)\circ m-h\in \widetilde{H}(R)$. 

Let $m=1+ X$, where $X\in \widetilde{T}(R)$.
Then $h\circ m-h=\widetilde{\varphi_R}(X)+\widetilde{\psi_R}(X)$.

Let $f=\widetilde{\psi_R}(Y)$ for some $Y\in \widetilde{T}(R)$.
Notice that we can always find such $Y$
since $\widetilde{\psi_R} : \widetilde{T}(R)\rightarrow \widetilde{H}(R)$ is surjective.
Then
$f\circ m =\widetilde{\psi_R}(Y)+\widetilde{\psi_R}(X^{ad}\cdot Y)+\widetilde{\psi_R}(Y\cdot X)+\widetilde{\psi_R}(X^{ad}\cdot Y\cdot X)$.
By Theorem 3.9 together with the fact that $\widetilde{T}(R)=R\otimes_A\widetilde{T}(A)$, $X^{ad}\cdot Y$, $Y\cdot X$ and $ X^{ad}\cdot Y\cdot X \in \widetilde{T}(R)$.
This completes the proof.
\end{proof}

 \begin{Thm}
  Let $\rho$ be the morphism $\underline{M}^{\ast} \rightarrow \underline{H}$ defined by $\rho(m)=h \circ m$.
  Then $\rho$ is smooth of relative dimension dim $G$.
  \end{Thm}

  \begin{proof}
  The theorem follows from Theorem 5.5 in \cite{GY} and the following lemma.
  \end{proof}

      \begin{Lem}
      The morphism $\rho \otimes \kappa : \underline{M}^{\ast}\otimes \kappa \rightarrow \underline{H}\otimes \kappa$
      is smooth of relative dimension $G$.
      \end{Lem}

\begin{proof}
    The proof is based on Lemma 5.5.2 in \cite{GY}.
    It is enough to check the statement over the algebraic closure $\bar{\kappa}$ of $\kappa$.
    By \cite{H}, III.10.4, it suffices to show that, for any $m \in \underline{M}^{\ast}(\bar{\kappa})$,
    the induced map on the Zariski tangent space $\rho_{\ast, m}:T_m \rightarrow T_{\rho(m)}$ is surjective.

We identify $T_m$ with $\widetilde{T}(\bar{\kappa})$ and $T_{\rho(m)}$ with $\widetilde{H}(\bar{\kappa})$.
Let $m=1+Y$ for some $Y\in \widetilde{T}(\bar{\kappa})$.
   The map $\rho_{\ast, m}:T_m \rightarrow T_{\rho(m)}$ is then
   $$X \mapsto \widetilde{\psi_{\bar{\kappa}}}(X+Y^{ad}\cdot X)=\widetilde{\psi_{\bar{\kappa}}}(m^{ad}\cdot X).$$
Since  $\widetilde{\psi_{\bar{\kappa}}} : \widetilde{T}(\bar{\kappa}) \longrightarrow \widetilde{H}(\bar{\kappa})$ 
 is surjective by Theorem 3.14,
it suffices to show that the following map
$$\widetilde{T}(\bar{\kappa})\longrightarrow \widetilde{T}(\bar{\kappa}), ~~~~~X\mapsto m^{ad}\cdot X$$
is bijective.
To prove the above, it is enough to show that
$\widetilde{T}(\bar{\kappa})\longrightarrow \widetilde{T}(\bar{\kappa}), ~~~~~X\mapsto m^{ad}\cdot X$ is well-defined so that
its inverse map $X \mapsto  (m^{-1})^{ad} \cdot X$ is  well-defined as well.
We consider a map $\underline{M}(R)\times \widetilde{T}(R)\longrightarrow \widetilde{T}(R), ~~~~~ (m, X)\mapsto m^{ad}\cdot X$, for a flat $A$-algebra $R$.
It is easy to show that this map is well-defined by using Theorem 3.9 together with the fact that $\widetilde{T}(R)=R\otimes_A\widetilde{T}(A)$ and  is represented by a morphism of schemes.
Therefore, a map $\underline{M}(\bar{\kappa})\times \widetilde{T}(\bar{\kappa})\longrightarrow \widetilde{T}(\bar{\kappa}), ~~~~~ (m, X)\mapsto m^{ad}\cdot X$
is well-defined as well.
 This completes the proof.
\end{proof}

 Let $\underline{G}$ be the stabilizer of $h$ in $\underline{M}^{\ast}$.
 It is  smooth and an affine subgroup scheme of $\underline{M}^{\ast}$, defined over $A$.
  Thus, we have the following theorem.
 \begin{Thm}
 The group scheme $\underline{G}$ with generic fiber $G$ is smooth,
 and $\underline{G}(R)=\mathrm{Aut}_{B\otimes_AR}(L\otimes_A R,h\otimes_A R)$ for any \'{e}tale $A$-algebra $R$.
 In conclusion, $\underline{G}$ is the desired smooth integral model of $G$.
 \end{Thm}
\begin{proof}
We choose an element $1+X\in \mathrm{Aut}_{B\otimes_AR}(L\otimes_A R,h\otimes_A R)$ for any \'{e}tale $A$-algebra $R$.
Then we only need to show that $1+X \in \underline{M}^{\ast}(R)$ and it suffices to show that $ X\in \widetilde{T}(R)$.
Firstly, it is clear that  $1+X\in T^0(R)$ so $X\in T^0(R)$.

Secondly, since $h\circ (1+X)=h$,
we have that
$\varphi_{0, R}(X)=\psi_{0, R}(-X).$
On the other hand,
$\varphi_{0, R}(X^{ad})=\psi_{0, R}(-X^{ad}).$
Therefore, by using induction, we conclude that $ X\in T^m(R)$ for all $m\geq 0$  so $ X\in \widetilde{T}(R)$.
\end{proof}


\section{Applications}
\subsection{Comparison of \cite{GY}, \cite{C1}, \cite{C2} and this paper}
As the first application, we explain how constructions of smooth integral models  in \cite{GY}, \cite{C1}, \cite{C2}
can be recovered from the construction studied in this paper, by finding $\alpha$ which is defined in the paragraph after Theorem 3.12.
We use terminology of \cite{C1} and \cite{C2} in the following $(2)$ and $(3)$, respectively.
\begin{enumerate}
\item If $p\neq 2$, then $\alpha=0$. Indeed, $\alpha=0$ for all cases covered in \cite{GY}.
\item Let $(K, \epsilon, p)=(F, 1, 2)$ and assume that $F/\mathbb{Q}_2$ is unramified, which is the case covered in \cite{C1}.
Then $\alpha$ is at most $2$.
Explicitly,
\begin{itemize}
\item if all Jordan components of $(L, h)$ are \textit{of type II}, then $\alpha=0$;
\item if there is a Jordan component \textit{free of type I} and if there is no Jordan component \textit{bound of type I},
then $\alpha=1$;
\item if there is a  Jordan component  \textit{bound of type I}, then $\alpha=2$.
\end{itemize}
\item Let $(K, \epsilon, p)=(E, 1, 2)$ and assume that $F/\mathbb{Q}_2$ is unramified and $E/F$ is ramified, which is the case covered in \cite{C2}.
Then $\alpha$ is at most $2$. Explicitly,
\begin{itemize}
\item when $E/F$ satisfies \textit{Case 1}, $\alpha=1$ if there is at least one Jordan component \textit{of type I}, and $\alpha=0$ otherwise.
\item When $E/F$ satisfies \textit{Case 2},
$\alpha=0$ if all Jordan components are \textit{of type II},
$\alpha=2$ if there is a Jordan component $L_i$ \textit{of type I} with $i$ even,
and $\alpha=1$ otherwise.
\end{itemize}
\end{enumerate}
In all cases, $\underline{M}^{\ast}$ and $\underline{H}$ (or $\underline{Q}$ in \cite{C1}) defined in \cite{GY}, \cite{C1}, \cite{C2}
are exactly the same as $\underline{M}^{\ast}$ and  $\underline{H}$ defined in this paper.

\subsection{The local density formula}
One major application of smooth integral models is a computation of local densities.
For a detailed explanation of a relation between local densities and smooth integral models, see the introduction of this paper and Section 3 of \cite{GY}.

As we have seen in Section 3.16, the construction of $\underline{G}$ simply follows from those of $\widetilde{T}$ and $\widetilde{H}$.
Furthermore, from the facts that $\widetilde{T}(R)=R\otimes_A\widetilde{T}(A)$ for a flat $A$-algebra $R$ (cf. Section 3.11),
 the only issue in the construction of $\underline{G}$ is $\widetilde{T}(A)$ and so $\underline{G}$ is constructed from some computation with $R=A$.

We briefly explain where $\widetilde{T}(A)$ and $\widetilde{H}(A)$ can be found in this paper.
Recall that we have the sequence
$$T^0(A)\supseteq T^1(A)\supseteq \cdots \supseteq T^m(A)\supseteq  \cdots$$
 at the paragraph before Theorem 3.12.
 Here, $T^0(A)$ is defined at the beginning of Section 3.2, $T^{1}(A)$ is defined at the beginning of Section 3.5,
 and $T^{m}(A)$ is defined at the beginning of Section 3.11.
 This sequence  stabilizes for a nonnegative integer $\alpha$ by Theorem 3.12
 and a method to find such $\alpha$ is explained at the paragraph after Theorem 3.12.
Finally $\widetilde{T}(A)=T^{\alpha}(A)$ at the paragraph after Theorem 3.12 and
$\widetilde{H}(A)$ is defined at the beginning of Section 3.13.

\begin{Thm}
Let $\tilde{G}$ be the special fiber of $\underline{G}$ and let $q$ be the cardinality of $\kappa$.
Let $M'=\mathrm{End}_{B}(L)$ and
$H'=
\{\textit{$f$ : $f$ is  hermitian form on $L$}\}$.
Notice that $M'/\widetilde{T}(A)$ and $H'/\widetilde{H}(A)$ are finitely generated torsion $A$-modules.
Let $$q^N=\frac{\#(H'/\widetilde{H}(A))}{\#(M'/\widetilde{T}(A))}.$$
Here, $\#$ stands for the cardinality and $N$ is an integer.
Then the local density of ($L,h$) is
$$\beta_L=\frac{1}{[G:G^o]}q^N \cdot q^{-\mathrm{dim~} G} \cdot \#\tilde{G}(\kappa),$$
where $G^o$ is the identity component of $G$.
\end{Thm}

Based on Theorem 4.3, once one constructs the smooth integral model $\underline{G}$ for a given lattice $(L, h)$,
the only real challenge to compute the local density is  $\#\tilde{G}(\kappa)$
and a recipe for computing $\#\tilde{G}(\kappa)$ is explained in the next section.

\section{Examples and a recipe for computing local densities}
In this section, in order to illustrate how effective and useful the recipe for computing local densities explained in this paper is,
we compute local densities of unimodular (i.e. $L^{\#}=L$) quadratic lattices with odd rank.
Then we explain a recipe for computing  local densities in the general case, based on the argument used in  the above example.

Let $A$ be a finite extension of $\mathbb{Z}_2$ and  $\pi$ be a uniformizing element of $A$.
Let $e=2e'$ or $e=2e'-1>1$ (we assume this for simplicity) be the ramification index with $(\pi)^e=(2)$ and  let $(L, h)$ be a quadratic lattice
(i.e. $K=F$ and $\epsilon=1$).
We use the symbol $A(a, b)$ to denote the $A$-lattice $A\cdot e_1+A\cdot e_2$ with the symmetric bilinear form having Gram matrix
$\begin{pmatrix} a&1\\ 1&b \end{pmatrix}$.
We denote by $(t)$ the  $A$-lattice of rank 1 equipped with the symmetric bilinear form having Gram matrix $(t)$.

Suppose that
$$L=\bigoplus_i A_i(0,0) \oplus A(\pi^s,r\pi^{2e-s})\oplus (t)$$
 of rank $2m+3$ with $r\in A$ 
and $t\equiv 1  \ \ \mathrm{mod}\ \ \pi$. Here, $s\leq e$ and $s$ is odd if $s<e$, and  $A_i(0,0)=A(0,0)$.
By Example 93:18 of \cite{O2}, $L$ as above exhausts all unimodular quadratic lattices with odd rank $>1$.
\subsection{Constructions of $\widetilde{T}$ and $\widetilde{H}$}
We first construct the smooth integral model $\underline{G}$ associated to $L$.
As mentioned in Section 4.2, the construction of $\underline{G}$ is based on the explicit description of $\widetilde{T}(A)$.
Basically, the procedure to find  $\widetilde{T}(A)$ gives certain congruence conditions on the entries of the matrix $\widetilde{T}(A)$.
We state matrix forms of $\widetilde{T}(A)$ and $\widetilde{H}(A)$ below.

\[
\widetilde{T}(A)=\left\{
  \begin{array}{l l}
 \begin{pmatrix} x_0&y_0&\pi^{[(e-s)/2]}z_0&\pi^{e'}w_0\\\pi^{[(e-s)/2]}x_1&\pi^{[(e-s)/2]}y_1&\pi^{e-s}z_1&\pi^{[e-(s-1)/2]}w_1\\
 x_2&y_2&\pi^{[(e-s)/2]}z_2&\pi^{e'}w_2\\\pi^{e'}x_3&\pi^{e'}y_3&\pi^{[e-(s-1)/2]}z_3&2w_3 \end{pmatrix} & \quad \textit{if $s<e$ odd};\\
 \begin{pmatrix} x&\pi^{e'}y \\ \pi^{e'}z& 2w  \end{pmatrix} & \quad \textit{if $s=e$}.
     \end{array} \right.
\]
Here, $x_0$ is a $(2m \times 2m)$-matrix, etc., and $x$ is a $(2m+2) \times (2m+2)$-matrix, etc., and
$[(e-s)/2]$ is the least integer greater than or equal to $(e-s)/2$.
Then a matrix form of $\underline{M}(A)$ is that of $\widetilde{T}(A)$ after replacing $\pi^{[(e-s)/2]}y_1$ by $1+\pi^{[(e-s)/2]}y_1$,
$\pi^{[(e-s)/2]}z_2$ by $1+\pi^{[(e-s)/2]}z_2$, and $2w_3$ by $1+2w_3$.

\[
\widetilde{H}(A)=\left\{
  \begin{array}{l l}
 \begin{pmatrix} a_0&b_0&\pi^{[(e-s)/2]}c_0&\pi^{e'}d_0\\{}^tb_0&2b_1&\pi^{[(e-s)/2]}c_1&\pi^{e'}d_1\\
 {}^t(\pi^{[(e-s)/2]}c_0)&{}^t(\pi^{[(e-s)/2]}c_1)&\pi^{2e-s}c_2&\pi^{[e-(s-1)/2]}d_2\\
 {}^t(\pi^{e'}d_0)&{}^t(\pi^{e'}d_1)&{}^t(\pi^{[e-(s-1)/2]}d_2)&4d_3 \end{pmatrix} & \quad \textit{if $s<e$ odd};\\
 \begin{pmatrix} a&\pi^{e'}b \\ {}^t(\pi^{e'}b)& 4c  \end{pmatrix} & \quad \textit{if $s=e$}.
     \end{array} \right.
\]
Here, the diagonal entries of $a_0$ are divisible by $2$, where $a_0$ is a $(2m \times 2m)$-matrix, etc., and
the diagonal entries of $a$ are divisible by $2$, where $a$ is a  $(2m+2) \times (2m+2)$-matrix, etc.
And ${}^tb_0$ is the matrix transpose of $b_0$.

Based on the above matrices, we can compute the $q^N$ factor in Theorem 4.3.
\begin{Lem}
Let $q^{N_1}=\#(M'/\widetilde{T}(A))$ and $q^{N_2}=\#(H'/\widetilde{H}(A))$ so that $q^N=q^{N_2-N_1}$.
Then we have the following:

$
\left\{  \begin{array}{l l}
N_1=2(2m+1)(e'+[(e-s)/2])+2[e-(s-1)/2]+2e-s & \quad \textit{ }\\
N_2=(2m+1)(e'+[(e-s)/2])+[e-(s-1)/2]+(2m+5)e-s & \quad \textit{if $s<e$ odd;}\\
  N=N_2-N_1=(2m+1)(e-e'-[(e-s)/2])-[e-(s-1)/2]+2e & \quad \textit{ }  \end{array} \right.
$

$
\left\{  \begin{array}{l l}
N_1=2(2m+2)e'+e & \quad \textit{ }\\
N_2=(2m+2)(e+e')+2e & \quad \textit{                  if $s=e$.}\\
  N=N_2-N_1=(2m+2)(e-e')+e & \quad \textit{ }   \end{array} \right.
$
\end{Lem}



\subsection{The special fiber $\tilde{G}$}
We define two sublattices $B(L), Z(L)$ of $L$ as follows:
\begin{itemize}
\item[(1)] $B(L)$, the sublattice of $L$ such that $B(L)/2L$ is the kernel of the additive polynomial $h$ mod 2 on $L/2L$.
\item[(2)] $Z(L)$,  the sublattice of $B(L)$ such that $Z(L)/\pi B(L)$ is the kernel of the quadratic form $\frac{1}{2}h$ mod $\pi$ on $B(L)/\pi B(L)$.
\end{itemize}

Let $\bar{V}=B(L)/Z(L)$ be a $\kappa$-vector space and $\bar{h}$ denote the nonsingular quadratic form $\frac{1}{2}h$ mod $\pi$ on $\bar{V}$.
It is  obvious that each element of $\underline{G}(A)$ fixes $\bar{h}$.
Since $\underline{G}$ is smooth, the map $\underline{G}(A)\rightarrow \tilde{G}(\kappa)$ is surjective by \textit{Hensel's lemma}.
Now, we choose an element $g\in \tilde{G}(\kappa)$ and its lifting $\tilde{g} \in \underline{G}(A)$.
Since $\tilde{g}$ induces an element of $\mathrm{O}(\bar{V}, \bar{h})(\kappa)$,
 we have a group homomorphism $\varphi_{\kappa}$ from $\tilde{G}(\kappa)$ to $\mathrm{O}(\bar{V}, \bar{h})(\kappa)$.
 It is easy to see that this map is well-defined, i.e. independent of a lifting $\tilde{g}$ of $g$.


\begin{Thm}
The group homomorphism $\varphi_{\kappa}$ defined by
 $$\varphi_{\kappa} : \tilde{G}(\kappa) ~ \longrightarrow  ~ \mathrm{O}(\bar{V}, \bar{h})(\kappa)$$
 is surjective.
$\mathrm{Ker~}\varphi_{\kappa} $ is isomorphic to $ (\textbf{A}^{l}\times (\mathbb{Z}/2\mathbb{Z})^{\beta})(\kappa)$ as sets,
 where $\textbf{A}^{l}$ is an affine space of dimension $l$.
 Here,
$l=\frac{(2m+3)(2m+2)}{2}-\textit{dimension of }\mathrm{O}(\bar{V}, \bar{h})$, and
$\beta=2$ (resp. $\beta=1$) if $s<e$ is odd   (resp. if $s=e$).
Therefore,
 $$ \# \tilde{G}(\kappa)= \#(\mathrm{O}(\bar{V}, \bar{h})(\kappa))\cdot q^l\cdot 2^{\beta}.$$
 Here $q$ is the cardinality of $\kappa$.
\end{Thm}

\begin{proof}
We  represent the given quadratic form $h$ by a symmetric matrix  $\delta=\begin{pmatrix}  \delta^{\prime}&0\\0&\delta^{\prime\prime}  \end{pmatrix}$,
 where $\delta^{\prime}$ is  $(2m \times 2m)$-matrix (resp.  $(2m+2) \times (2m+2)$-matrix)
 if $s<e$ (resp. $s=e$).

Let $\tilde{M}$ be the special fiber of $\underline{M}^{\ast}$ and choose an element $g\in \tilde{G}(\kappa)$.
Recall that a formal matrix form of $g$, as an element of $\tilde{M}(\kappa)$, is
\[
g=\left\{
  \begin{array}{l l}
 \begin{pmatrix} x_0&y_0&\pi^{[(e-s)/2]}z_0&\pi^{e'}w_0\\\pi^{[(e-s)/2]}x_1&1+\pi^{[(e-s)/2]}y_1&\pi^{e-s}z_1&\pi^{[e-(s-1)/2]}w_1\\
 x_2&y_2&1+\pi^{[(e-s)/2]}z_2&\pi^{e'}w_2\\\pi^{e'}x_3&\pi^{e'}y_3&\pi^{[e-(s-1)/2]}z_3&1+2w_3 \end{pmatrix} & \quad \textit{if $s<e$ odd};\\
 \begin{pmatrix} x&\pi^{e'}y \\ \pi^{e'}z& 1+2w  \end{pmatrix} & \quad \textit{if $s=e$}.
     \end{array} \right.
\]
We note that the above matrix description is valid for elements of $\underline{G}(A)$, as elements of $\underline{M}(A)$
where $\underline{M}(A)$ is defined in Section 3.16.
However, as mentioned in Section 5.3 of \cite{GY}, we formally use this matrix description for $g\in \tilde{G}(\kappa)$.
To multiply $g$ and $g'$ for $g, ~g'\in \tilde{G}(\kappa)$, we refer to the description of Section 5.3 of \cite{GY}.
Then under the morphism $\varphi_{\kappa}$, $g$ maps to
\[\left\{
  \begin{array}{l l}
 \begin{pmatrix} x_0&0\\x_1 \textit{ (resp. $x_3$)}&1 \end{pmatrix} & \quad \textit{if $s<e$ odd and $e=2e'-1$ (resp. $e=2e'$)};\\
 \begin{pmatrix} x \end{pmatrix} ~ \left(\textit{resp. }\begin{pmatrix} x&0 \\ z&1\end{pmatrix}\right) &  \quad \textit{if $s=e$  and $e=2e'-1$ (resp. $e=2e'$)}.
     \end{array} \right.
\]
We define the subgroup $H$ of $\tilde{G}(\kappa)$ by setting as follows:
\[
\left\{
  \begin{array}{l l}
y_0=w_0=0, x_2=x_3=0, y_1=w_1=y_2=z_2=w_2=y_3=z_3=w_3=0  & \quad \textit{$e=2e'-1$ and $s<e$ odd};\\
  y=0, z=0, w=0 & \quad \textit{$e=2e'-1$ and $s=e$};\\
y_0=z_0=0, x_1=x_2=0, y_1=z_1=w_1=y_2=z_2=w_2=y_3=z_3=0  & \quad \textit{$e=2e'$ and $s<e$ odd};\\
\textit{no contribution}  & \quad \textit{$e=2e'$ and $s=e$}.
     \end{array} \right.
\]
Then by observing equations defining $H$, we can easily show that the restriction map $\varphi_{\kappa}|_{H}$ is surjective,
which induces the surjectivity of  $\varphi_{\kappa}$.
Note that equations defining $H$, as a subgroup of $\tilde{M}(\kappa)$,
can be obtained by observing each block of the formal matrix equation ${}^t g \delta g=\delta$, where $g \left(\in \tilde{M}(\kappa)\right)$ is as above setting.
  Here,  the sum and the multiplication in the formal matrix equation are to be interpreted as in Section 5.3 of \cite{GY}.

In addition,  $\mathrm{Ker~}\varphi_{\kappa}$ is obtained, as a subgroup of  $\tilde{G}(\kappa)$, by setting:
\[
\left\{
  \begin{array}{l l}
x_0=id, x_1=0  & \quad \textit{$e=2e'-1$ and $s<e$ odd};\\
  x=id & \quad \textit{$e=2e'-1$ and $s=e$};\\
x_0=id, x_3=0   & \quad \textit{$e=2e'$ and $s<e$ odd};\\
x=id, z=0  & \quad \textit{$e=2e'$ and $s=e$}.
     \end{array} \right.
\]

  We now state the equations defining $\mathrm{Ker~}\varphi_{\kappa}$, as a subgroup of $\tilde{M}(\kappa)$, by observing each block of the formal matrix equation ${}^t g \delta g=\delta$  as follows.
  Again,  the sum and the multiplication in the formal matrix equation are to be interpreted as in Section 5.3 of \cite{GY}.

If $e=2e'-1$ and $s<e$ is odd,
\begin{align*}
 \delta'y_0+{}^tx_2=0, z_0=0, \delta^{\prime}w_0+{}^tx_3=0, {}^ty_0\delta^{\prime}y_0/2+y_1^2+y_2=0, z_2+{}^ty_1=0, \\
 {}^ty_0\delta^{\prime}w_0+w_2+{}^ty_3=0, z_1^2+z_1=0, w_1+z_3=0, w_3^2+w_3=0.
\end{align*}

If $e=2e'-1$ and $s=e$,
\begin{align*}
\delta^{\prime}y+{}^tz=0, w^2+w=0. \end{align*}

If $e=2e'$ and $s<e$ is odd,

\begin{align*}
  \delta'y_0+{}^tx_2=0, \delta'z_0+{}^tx_1=0, w_0=0, {}^ty_0\delta^{\prime}y_0/2+y_2+\pi^{e}/2y_3^2=0, \\
{}^ty_0\delta'z_0+z_2+{}^ty_1=0, w_2+{}^ty_3=0, z_1^2+z_1=0, w_1+z_3=0, w_3^2+w_3=0.
\end{align*}

If $e=2e'$ and $s=e$,
\begin{align*}
y=0, w^2+w=0.
\end{align*}

The claim  regarding $\mathrm{Ker~}\varphi_{\kappa} $ directly follows by observing the above equations.
\end{proof}

\begin{Rmk}
 We describe Im $\varphi_{\kappa}$ as follows.
     \[
      \begin{array}{c|c}
       & \mathrm{Im~}  \varphi_{\kappa} \\
      \hline
      \textit{$e=2e'-1$ and $s<e$ odd}\ \   & \mathrm{O}(2m+1, \bar{h})(\kappa)=\mathrm{SO}(2m+1, \bar{h})(\kappa)\\
      \textit{$e=2e'-1$ and $s=e$}\ \   &\mathrm{O}(2m+2, \bar{h})(\kappa)\\
      \textit{$e=2e'$ and $s<e$ odd}\ \   &\mathrm{O}(2m+1, \bar{h})(\kappa)=\mathrm{SO}(2m+1, \bar{h})(\kappa)\\
      \textit{$e=2e'$ and $s=e$}\ \   &\mathrm{O}(2m+3, \bar{h})(\kappa)=\mathrm{SO}(2m+3, \bar{h})(\kappa)\\
            \end{array}
    \]
    \end{Rmk}
\begin{Thm}
By combining Lemma 5.2 and Theorem 5.4 with Theorem 4.3, the local density of ($L,h$) is
$$\beta_L=1/2 \cdot q^N \cdot q^{-\mathrm{dim~} G} \cdot \#\tilde{G}(\kappa)=
2^{\beta-1}\cdot q^{N-\mathrm{dim~} \mathrm{O}(\bar{V}, \bar{h})}\cdot \#(\mathrm{O}(\bar{V}, \bar{h})(\kappa)),$$
where
$$N=\left\{
  \begin{array}{l l}
  (2m+1)(e-e'-[(e-s)/2])-[e-(s-1)/2]+2e & \quad \textit{if $s<e$ odd};\\
  (2m+2)(e-e')+e & \quad \textit{if $s=e$}.
     \end{array} \right.$$
Here,  $\beta$ is defined in Theorem 5.4 and $\mathrm{O}(\bar{V}, \bar{h})(\kappa)$ is as explained in Remark 5.5.
And $q$ is the cardinality of $\kappa$.

\end{Thm}

\begin{Rmk}
\begin{itemize}
\item[(1)] As promised in the introduction,
we  use less scheme language in the above example. All computations can be done over finitely generated $A$-modules.

\item[(2)] Theorem 5.4 can be generalized to give an algebraic group structure of $\tilde{G}$ by working with $\kappa$-algebras' points.
Indeed, $\tilde{G}$ satisfies the short exact sequence
\[1 \longrightarrow \mathrm{Ker~}\varphi \longrightarrow \tilde{G} \overset{\varphi}\longrightarrow \mathrm{O}(\bar{V}, \bar{h})^{\mathrm{red}} \longrightarrow 1 \]
where $\mathrm{Ker~}\varphi \cong \textbf{A}^{l}\times (\mathbb{Z}/2\mathbb{Z})^{\beta}$ as varieties over $\kappa$
and $\mathrm{O}(\bar{V}, \bar{h})^{\mathrm{red}}$ is the reduced subgroup scheme of $\mathrm{O}(\bar{V}, \bar{h})$.
For a detailed or scheme-theoretic explanation of the construction of the morphism $\varphi$, we refer to Section 4.1 of \cite{C1}.
One can also construct a morphism from $\tilde{G}$ to $(\mathbb{Z}/2\mathbb{Z})^{\beta}$ as algebraic groups and the construction of this morphism
is similar to that of Sections 4.2 and 4.3 of \cite{C1}.
Thus, we obtain the following short sequence
\[1 \longrightarrow R_u\tilde{G} \longrightarrow \tilde{G} \longrightarrow \mathrm{O}(\bar{V}, \bar{h})^{\mathrm{red}}\times (\mathbb{Z}/2\mathbb{Z})^{\beta} \longrightarrow 1 \]
where the unipotent radical $R_u\tilde{G} $ is isomorphic to $\textbf{A}^{l}$ as varieties over $\kappa$.

However, we only need $\# \tilde{G}(\kappa)$ in order to compute the local density.
Thus, for our purpose it is enough to work with $\kappa$-points rather than an algebraic group structure.
\end{itemize}
\end{Rmk}
\begin{Rmk}
 We compare our formula with the formula of Conway-Sloane (\cite{CS}).
The definition of local densities used in this paper is slightly different from that of \cite{CS}.
In this paper, we divide the definition of \cite{CS} by $2$ in Theorem 4.3, where $2$ is the number of connected components of an orthogonal group.
In the following, we use the definition of local densities used in \cite{CS} in order to compare our formula with theirs.

It seems that it is  hard to identify the contributions that occur in our formula with the factors appearing in the Conway-Sloane result.
Let us explain this by observing the simplest example, namely a unimodular lattice $L$ with rank $1$ over $\mathbb{Z}_2$.

In this case, $$\widetilde{T}_1(A)=(2x) \textit{~and~} \widetilde{T}_2(A)=(4a)$$
and so the local density studied in this paper is $$2\cdot q^e = 2\cdot q ~(=4).$$
Here, $2=\# (\tilde{G}(\kappa))=\# (\mathbb{Z}/2\mathbb{Z})$  (i.e. $\beta=1$)
and $N=e=1$.

Based on \cite{CS}, there are three nontrivial Jordan components: $L=L_{-1}\oplus L_0 \oplus L_1$, where
$L_{-1}$ and $L_1$ are \textit{bound love forms of rank $0$} (section 6 of \cite{CS})
and $L_0$ is  \textit{free of type $I$ with rank $1$}.
Thus, \[\textit{the local density of $L$ studied in \cite{CS}=1/(the p-mass of L)=$2\cdot 2\cdot 1=4$}.\]
Here, the first two $2$'s are the reciprocal of the diagonal factors associated to $L_{-1}$ and $L_1$,
and $1$ is the reciprocal of the diagonal factor associated to
$L_0$. (Sections 5,6, and 12 of \cite{CS})

In order to match our formula with their formula, there should be correspondence between  $2\cdot q$ and $2\cdot 2\cdot 1$.
But we do not see which $2$ in their formula should correspond to $2$ in our formula (or correspond to $q$) since their two $2$'s come from \textit{bound love forms}.
In other words, we do not see how we can distinguish their two $2$'s.

For this reason, it is not clear to us how to generalize their formula to  unramified/ramified cases.
\end{Rmk}

\subsection{A recipe for computing local densities}
In this section, we explain a recipe for computing local densities.
Let $L=\bigoplus_{i\geq 0}L_i$ be a Jordan splitting for a quadratic lattice $(L, h)$.
We define three sublattices $A_i, B_i, Z_i$ of $L$ as follows:
\begin{itemize}
\item[(1)] $A_i$, $A_i=\{x\in L:h(x, L)\in \pi^iA\};$
\item[(2)] $B_i$, the sublattice of $L$ such that $B_i/2 A_i$ is the kernel of the additive polynomial $1/\pi^i \cdot h$ mod 2 on $A_i/2 A_i$;
\item[(3)] $Z_i$,  the sublattice of $B_i$ such that $Z_i/\pi B_i$ is the kernel of the quadratic form $1/2\cdot 1/\pi^i \cdot h$ mod $\pi$ on $B_i/\pi B_i$.
\end{itemize}
Let $\bar{V_i}=B_i/Z_i$ be a $\kappa$-vector space and $\bar{h_i}$ denote the nonsingular quadratic form $1/2\cdot 1/\pi^i \cdot h$ mod $\pi$ on $\bar{V_i}$.
Then, as the group homomorphism $\varphi_{\kappa}$ defined in Section 5.3, we have a group homomorphism
$\varphi_{i, \kappa} : \tilde{G}(\kappa) \rightarrow \mathrm{O}(\bar{V_i}, \bar{h_i})(\kappa).$
Let $$\varphi_{\kappa}=\prod_i \varphi_{i, \kappa} : \tilde{G}(\kappa) \rightarrow \prod_i \mathrm{O}(\bar{V_i}, \bar{h_i})(\kappa).$$


\begin{Conj}
We  expect  that
\[\textit{$\varphi_{\kappa}$ is surjective and its kernel is isomorphic to $(\textbf{A}^{l}\times (\mathbb{Z}/2\mathbb{Z})^{\beta})(\kappa)$ as sets,}\]
for $l=\textit{dimension of $\tilde{G}$ - $\sum_i$ (dimension of $\mathrm{O}(\bar{V_i}, \bar{h_i})$)} $ and for a certain non-negative integer $\beta$.
\end{Conj}
Note that the dimension of $\tilde{G}$ is the same as the dimension of $G$ which is $n(n-1)/2$, where $G$ is the generic fiber of $\underline{G}$ and $n$ is the rank of $L$, since $\underline{G}$ is flat over $A$.
Once one verifies this  conjecture, we obtain that
$$\# \tilde{G}(\kappa)=\prod_i\#(\mathrm{O}(\bar{V_i}, \bar{h_i})(\kappa))\cdot q^l\cdot 2^{\beta}.$$
Here, $q$ is the cardinality of $\kappa$.
Then, combined with Theorem 4.3, one obtains the local density.

\begin{Rmk}
This conjecture is proved  when $A/\mathbb{Z}_2$ is unramified (\cite{C1}) or when $(L, h)$ is unimodular of odd rank  (Theorem 5.4)
and the proof is based on formal matrix interpretations of an element of $\tilde{G}(\kappa)$ and the group homomorphism $\varphi_{\kappa}$.
In the general case, the  classification of quadratic (or hermitian) lattices over a finite ramified extension of $\mathbb{Z}_2$ is a highly involved problem  (\cite{J}, \cite{O1}, \cite{O2}) and so it seems unlikely that  formal matrix forms of an element of $\tilde{G}(\kappa)$ and the group homomorphism $\varphi_{\kappa}$
can be written.
For this reason,  it seems unlikely that the conjecture is proved in the general case.

However,  we expect that one can still verify the conjecture case by case, and a possible framework to prove it is given below.
For a given specific lattice,  
one can write an explicit matrix form of an element of $\underline{G}(A)$ which allows one to
 write  formal matrix forms of an element of $\tilde{G}(\kappa)$ and the group homomorphism $\varphi_{\kappa}$ explicitly, as in the proof of Theorem 5.4.
Based on formal matrix interpretations of these, one can then prove surjectivity directly by observing equations defining $\tilde{G}(\kappa)$ and
a formal matrix interpretation of the group homomorphism $\varphi_{\kappa}$,
or  one may choose a certain suitable subgroup $H$ of $\tilde{G}(\kappa)$ as in Theorem 5.4 and Example 5.12
 such that the restriction of $\varphi_{\kappa}$ to $H$ is surjective.
Then, again based on a formal matrix interpretation of $\varphi_{\kappa}$,
one can enumerate all equations defining $\mathrm{Ker~}\varphi_{\kappa}$ and check that
$\mathrm{Ker~}\varphi_{\kappa}$ is isomorphic to $(\textbf{A}^{l}\times (\mathbb{Z}/2\mathbb{Z})^{\beta})(\kappa)$ as sets.
\end{Rmk}


Below, we give an example to compute the local density for a non-unimodular quadratic lattice by proving the conjecture.
\begin{Exa}
Assume that the ramification index $e$ is $2$  
 and $L=\left(\oplus_i A_i(0,0)\oplus (1)\right)\oplus \left(\oplus_j \pi A_j(0,0)\right) $ of rank $2m+1+2m'$.
Here, $\pi A_j(0,0)$ is a quadratic lattice with Gram matrix $\begin{pmatrix} 0&\pi\\ \pi&0 \end{pmatrix}$.
So  a Jordan splitting of $L$ is $L=L_0\oplus L_1$ and
we represent the given quadratic form $h$ by a symmetric matrix $\begin{pmatrix} \delta&0&0\\0&1&0\\0&0&\pi \delta' \end{pmatrix}$,
where $\delta$ is a $(2m \times 2m)$-matrix  and $\delta'$ is a  $(2m' \times 2m')$-matrix.
Then we have the following description for $\widetilde{T}(A)$ and $\widetilde{H}(A)$:
\[\widetilde{T}(A)=\begin{pmatrix}x_0&\pi y_0&\pi z_0\\\pi x_1&\pi^2 y_1&\pi^2 z_1\\x_2&\pi y_2&z_2\end{pmatrix}, ~~
\widetilde{H}(A)=\begin{pmatrix}a_0&\pi b_0&\pi c_0\\\pi\cdot {}^t b_0&4 b_1&\pi^2 c_1\\\pi\cdot {}^t c_0&\pi^2\cdot {}^tc_1&\pi c_2\end{pmatrix}. \]
Here, $x_0$ and $a_0$ are $(2m \times 2m)$-matrices,  and $z_2$ and $c_2$ are $(2m' \times 2m')$-matrices, etc.
In addition, the diagonal entries of $a_0$ and $c_2$ are divisible by $2$.
Then we can compute $q^N$ in Theorem 4.3 such that
$$N=(2(m')^2+4m m'+6m+9m'+4)-(4m+6m'+4m m'+2)=2(m')^2+2m+3m'+2.$$

To compute $\# \tilde{G}(\kappa)$, we describe
$\varphi_{\kappa}=\varphi_{0, \kappa}\times \varphi_{1, \kappa} : \tilde{G}(\kappa) \longrightarrow \mathrm{O}(\bar{V_0}, \bar{h_0})(\kappa)\times \mathrm{O}(\bar{V_1}, \bar{h_1})(\kappa)$
explicitly in terms of  formal matrices.
Recall that a formal matrix form of an element $g\in \tilde{G}(\kappa)$, as an element of $\tilde{M}(\kappa)$, is
$\begin{pmatrix}x_0&\pi y_0&\pi z_0\\\pi x_1&1+\pi^2 y_1&\pi^2 z_1\\x_2&\pi y_2&z_2\end{pmatrix}$.
Then under the morphism $\varphi_{\kappa}$, $g$ maps to
$\begin{pmatrix}x_0&0\\x_1&1\end{pmatrix} \times \begin{pmatrix}z_2\end{pmatrix}$.

If we define the subgroup $H$ of $\tilde{G}(\kappa)$ by setting
$z_0=0, y_1=0, z_1=0, x_2=0, y_2=0$, then we can easily show that
$\varphi_{\kappa}|_{H}$ is surjective by observing equations defining $H$, which induces the surjectivity of $\varphi_{\kappa}$.
Note that equations defining $H$, as a subgroup of $\tilde{M}(\kappa)$,
can be obtained by observing each block of the formal matrix equation ${}^t g \delta g=\delta$, where $g \left(\in \tilde{M}(\kappa)\right)$ is as above setting.
  Here,  the sum and the multiplication in the formal matrix equation are to be interpreted as in Section 5.3 of \cite{GY}.

In addition, based on the formal matrix equation ${}^t g \delta g=\delta$, the kernel of $\varphi_{\kappa}$ is defined by the following equations:
\[ \delta y_0=0, \delta z_0+{}^tx_2 \delta'=0, \bar{u}y_1+(\bar{u}y_1)^2=0, z_1+{}^ty_2\delta'=0. \]
Here, $\bar{u}$ is such that $\pi^2=2u$ for a unit $u\in A$ and $\bar{u}$ is the reduction of $u$ in $\kappa$.
Thus, $ \mathrm{Ker~}\varphi_{\kappa}=(\textbf{A}^{l}\times (\mathbb{Z}/2\mathbb{Z})^{1})(\kappa)=(\textbf{A}^{4m m'+2m'}\times \mathbb{Z}/2\mathbb{Z})(\kappa)$, i.e. $l=4m m'+2m'$ and $\beta=1$.
Therefore Conjecture 5.10 is verified for this example.

In conclusion, $\# \tilde{G}(\kappa)=\#(\mathrm{O}(2m+1, \bar{h_0})(\kappa))\cdot \#(\mathrm{O}(2m', \bar{h_1})(\kappa))\cdot q^{4m m'+2m'}\cdot 2$
and based on Theorem 4.3, the local density of $(L, h)$ is
$$\beta_L=1/2 \cdot q^N \cdot q^{-\mathrm{dim~} G} \cdot \#\tilde{G}(\kappa)=q^{m+4m'+2-2m^2}\cdot \#(\mathrm{O}(2m+1, \bar{h_0})(\kappa))\cdot \#(\mathrm{O}(2m', \bar{h_1})(\kappa)).$$
\end{Exa}
\textit{}

For other lattices such as hermitian lattices, one can still construct a group homomorphism
$$\varphi_{\kappa}=\prod_i \varphi_{i, \kappa} : \tilde{G}(\kappa) \rightarrow \prod_i \textit{(classical group)}.$$
Note that the paper \cite{C2} gives an explicit construction of such morphism for ramified hermitian lattices defined over an unramified extension of
$\mathbb{Z}_2$. Then one can set up the conjecture as Conjecture 5.10 that
\[\textit{$\varphi_{\kappa}$ is surjective and its kernel is isomorphic to $(\textbf{A}^{l}\times (\mathbb{Z}/2\mathbb{Z})^{\beta})(\kappa)$, as sets.}\]
 By proving this as similar to Theorem 5.4 and Remark 5.11 case by case,  $\# \tilde{G}(\kappa)$ can be computed explicitly
and accordingly, the local density is obtained.

\end{document}